\documentclass[a4paper,12pt,fleqn]{article}
\usepackage[latin1]{inputenc}
\usepackage[T1]{fontenc}
\usepackage[a4paper,lmargin=4.1cm,textwidth=12.8cm,tmargin=4.9cm,textheight=18.5cm]{geometry}
\usepackage[all,2cell,v2]{xy}
\usepackage[all]{xy}
\UseAllTwocells
\usepackage[square,numbers,sort]{natbib}
\usepackage{ifthen}
\usepackage{mathbbol}
\usepackage{savesym}
\usepackage{amsmath}
\usepackage{amssymb}
\savesymbol{iint}
\restoresymbol{TXF}{iint}
\usepackage{mathptmx}
\usepackage[thmmarks,standard]{ntheorem}
\usepackage{microtype}
\usepackage{stmaryrd}
\usepackage[english]{babel}
\newboolean{PourEditeur}
\setboolean{PourEditeur}{false}
\ifthenelse{\boolean{PourEditeur}}{%
}{%
  \usepackage{hyperref}
}
\usepackage{cleveref}

\usepackage[toc,page]{appendix}

\theoremstyle{plain}

\newtheorem{result}{Result}

\crefformat{result}{result~#2#1#3}
\crefformat{proposition}{proposition~#2#1#3}
\crefformat{remark}{remark~#2#1#3}
\newcommand*{\G}{\ensuremath{\mathbb{G}\text{r}}}
\newcommand*{\IG}{\ensuremath{\infty\text{-}\G}}
\newcommand*{\C}{\ensuremath{\mathbb{C}\text{AT}}}
\newcommand*{\TC}{\ensuremath{T\text{-}\C}}

\DeclareMathOperator{\Mo}{Mo}
\DeclareMathOperator{\colim}{colim}
\title{Overcategories and free monoids for overcategories}
\author{Camell Kachour}
\begin{document}
\maketitle
\begin{abstract}

An overcategory with base category $\mathcal{C}$ is merely any functor into $\mathcal{C}$.
 In this paper we extend the work of Dominique Bourn and Jacques Penon \cite{bournpenon:moncart} on overcategories. In particular we show that Freyd's adjoint theorem, a theorem of Barr and Wells in \cite{barrwells:ttt} are still valid in the overcategorical context. We also show that a free monoid 
  construction remains valid in the context of overcategories. The motivation for this study is the 
   development of higher categories as found in \cite{bournpenon:moncart} and in 
  \cite{kach3:redmacq}.

 \end{abstract}
\tableofcontents
\vspace{1cm}

 Overcategories are just objects of the comma category
 $(\C\downarrow \mathcal{C})$ where $\mathcal{C}\in \C$ without any other 
 requirements and in \cite{Street_Schuma} the authors call it parametrized categories. 
 Thus, as categorical structure, they are poorer than indexed categories
 or fibrations. It is well known that we can extend most categorical concepts to
  indexed categories and fibrations, and most of the usual theorems for categories
  (for instance, Freyd's adjoint theorem or Beck's monadicity theorem)
  have their equivalent for indexed categories and fibrations 
  (for indexed categories, see \cite{Pare_Schuma}). Supprisingly,
  they are few studies for overcategories despite the fact that we find applications
  of it for example in \cite{bournpenon:moncart} and in \cite{kach3:redmacq}, 
  where both of these works deal with the perspective of higher category theory. 
  
  Let us explain briefly the differences between our study of overcategories and others important
  approaches as we can find for example in \cite{Pare_Schuma} for the indexed categories or in 
  \cite{jean_benabou1985, Street_Schuma} for the fibrations. In \cite{Pare_Schuma} the authors have developped intensively the theory of indexed categories, which are fibrations with a choice of a cleavage (or if we use the "Grothendieck construction", it is a pseudofunctor with a choice of isomorphisms in Cat). In these very precise context they proved an "Adjoint Functor Theorem" á la Peter Freyd. They used an "Initial object theorem" to prove 
  it, but all the time in the context of their "Indexed categories". In our approach we prove also an
  "Adjoint Functor Theorem" a la Peter Freyd and an "Initial object theorem" to prove 
  it, but in the poorer context of overcategories. In \cite{jean_benabou1985} J.B{\'e}nabou use the theory of fibrations with a concept of "definability" for fibrations, and he tried to make more much clearer the problem of finding the good "logical environment" to build category theory. For him, category theory can be build with the notion
of fibration and definability. An other aspect of the paper \cite{jean_benabou1985} is the fact that the author said that "Indexed categories" are not the good environment to build category theory, because the choice of cleavage (using the choice axiom), which is a part of the definition of an "Indexed categories", makes things much more harder and less natural, where many confusions can be made. Thus the spirit of the paper of 
Benabou is completely different from our approach where we just show that we can developp some concepts of category theory for overcategories rather than to be embarrased by the foundation of category theory itself. In the
beginning of the paper \cite{Street_Schuma} the authors start to speak about overcategories that they called 
"parametrized categories", however then they "enriched" quickly this notion with the notion of "cartesian cones", "cartesian maps" (which is here a specific case of the cartesian cones), and fibrations. In all their paper they studied properties like "wellpoweredness", "small idempotency", "generators", etc. which are the major new notions of this paper, but with results involving at less cartesian maps or more, fibrations. In particular they gave a new characterisation of an elementary toposes with their concept of wellpowerdness for parametrized categories. But despite of all these very interesting results, their paper has different perspectives from our approach where we
investigate overcategories without any others structures on it.
  
 Dominique Bourn and Jacques Penon have used overcategories 
 in \cite{bournpenon:moncart} (called "surcat{\'e}gories" in their article) as a major tool for their 
 studies of categorification, which is one of the most important question in higher category theory.
  
 However we believe that the level of generality of our study of overcategories in this article and in
 the article \cite{bournpenon:moncart}, allows the possibility of applications to many other contexts, especially
 those where the categories involved have structure poorer than fibrations or indexed categories. Let us be
 more precise on this point. For example if for a fixed category $\mathcal{C}$ it is difficult to see that it is complete or
 cocomplete, but this category is equipped with a functor $\xymatrix{\mathcal{C}\ar[r]^{F}&\mathcal{D}}$ such 
 that the fibers $F^{-1}(d) (d\in\mathcal{D}$) are complet or cocomplet and this (co)completness 
 has some good properties among the whole category $\mathcal{C}$, then we can ask if this underlying
 structure of $\mathcal{C}$ is good enough to resolve some mathematical questions which are involved. This is
 exactly the first motivation of the author in \cite{kach3:redmacq} to study 
 overcategories because at that time he found that the cocompleteness
 of the category $\TC$ of $\mathbb{T}$-categories was not evident.    
  
 The first section of this article deals with the overcategorical version of 
 some classical theorems of the category theory.
 The "overcategorical theorems" that we establish in this paper, especially the overadjoint Freyd theorem, and a theorem overcategoric of Barr-Wells theorem, could be useful for classical category theory itself.
  
The second section of this article deals with the free monoid construction in the context
of overcategories. We build it within framework of the overmonoidal
overcategories\cite{bournpenon:moncart}. As a matter of fact
in \cite{bournpenon:moncart} the authors establish an adjunction
result (to obtain free "overmonoids") in an ideal context they label
"numeral" \cite{bournpenon:moncart}.  We
demonstrate a similar theorem (\cref{theoremdufreesurmonoid}) which
also results from an ideal context that I label liberal and which
allows us to establish a result of free overmonoids result

 I am grateful to Jacques Penon who permitted me to access the
details of his conjoint work \cite{bournpenon:moncart} with Dominique Bourn.
The research for this present paper was completed in 2009.

  \vspace*{1cm}   

  \section{Theory of Overcategories }
\label{TheoryofSurcategories}
Let $\mathbb{G}$ be a fixed category.  An overcategory is an object of the
$2$-category $\C/\mathbb{G}$. Thus it is given by a couple
$(\mathcal{C},\mathtt{A})$, where $\mathcal{C}$ is a category and
$\mathcal{C}\xrightarrow{\mathtt{A}}\mathbb{G}$ is a functor (often called "arity functor", in reference
to its use in this paper). In what follows the arity functor
is often noted with the letter $\mathtt{A}$ because there is no risk
of confusion.  The evident morphisms of $\C/\mathbb{G}$ are called
overfunctors, but we also need in this paper morphim between
overcategories with different base categories $\mathbb{G}$ and
$\mathbb{G}'$. Therefore such a morphism
$(\mathcal{C},\mathtt{A})\xrightarrow{(F,F_{0})}(\mathcal{C}',\mathtt{A}')$
is given by two functors: $\mathcal{C}\xrightarrow{F}\mathcal{C}'$ and
$\mathbb{G}\xrightarrow{F_{0}}\mathbb{G}'$ such as
$\mathtt{A}'F=F_{0}\mathtt{A}$ (see for example
\cref{FreydTheoremintheSurcategoricalcontext}).  For a fixed
overcategory $(\mathcal{C},\mathtt{A})$, its objects and its morphisms
are respectively objects and morphisms of the domain category
$\mathcal{C}$.

The pairs of adjoints morphisms and monads in the $2$-category
$\C/\mathbb{G}$ are called respectively pairs of adjoint
overfunctors and overmonads.  In fact every overcategorical concept will
be expressed by using "over" before the categorical concept
that it generalizes.  But we sometimes forget the word "over", when the
context implies that no confusion is possible.  It is easy to see that
the category of algebras for a given overmonad is an overcategory. The
objects of this overcategory will be called overalgebras.
 
We are going to see that most of the notions in the $2$-category $\C$
can be done again in the $2$-category $\C/\mathbb{G}$, and it is very
likely that most of the concepts and theorems in $\C$ extend to
$\C/\mathbb{G}$. We will particularly demonstrate three theorems in
$\C/\mathbb{G}$ coming from three important theorems in $\C$: Freyd's
Overadjoint Theorem (which is the overcategorical version of the classical Freyd
Adjoint Theorem. See \cref{surfreyd}), Barr-Wells's
Overcategorical Theorem (which is the overcategorical version of the
result that we can find in \cite{barrwells:ttt}. See
\cref{SurBarrWells}), and Beck's Overcategorical Theorem (which is the
overcategorical version of Beck's classical theorem. See
\cref{surbeck}).  These theorems are the obvious generalisations of
the classical ones.
\subsection{Definition of Over(co)limits}
\label{defsurlimits}
In \cite{bournpenon:moncart} the notions of limits and
colimits in $\C/\mathbb{G}$ are defined and these notions will be used
afterwards. To facilitate the reader we will recall the definitions.
 
If $\mathbb{C}$ is a small category and if $\mathcal{E}$ is a
category, then we have the classical diagonal functor
$\mathcal{E}\xrightarrow{\Delta}\mathcal{E}^{\mathbb{C}}$, which
sends an object to a constant functor and which sends a morphism to a
constant natural transformation.
 
Moreover if $(\mathcal{E},\mathtt{A})$ is a overcategory, let
$\mathcal{E}^{(\mathbb{C})}$ be the subcategory of
$\mathcal{E}^{\mathbb{C}}\times\mathbb{G}$ given by:
\begin{description}
\item $\mathcal{E}^{(\mathbb{C})}(0)=
  \{(F,B)\in\mathcal{E}^{\mathbb{C}}\times\mathbb{G}/
  \mathtt{A}F=\Delta(B)\}$,
\item $\mathcal{E}^{(\mathbb{C})}(1)=
  \{(F,B)\xrightarrow{(\tau,b)}(F',B')/ b\in
  \mathbb{G}(1)\hspace{.1cm}\text{and}\hspace{.1cm}
  \tau\hspace{.1cm}$ is a natural transformation such as
  $\hspace{.1cm}\mathtt{A}\tau=\Delta(b)\}$.
\end{description}

$\mathcal{E}^{(\mathbb{C})}$ has a natural overcategory structure given
by the second projection:
$\mathcal{E}^{(\mathbb{C})}\xrightarrow{\mathtt{A}}\mathbb{G}$,\hspace{.2cm}
$(F,B)\longmapsto B$. In fact
$(\mathcal{E}^{(\mathbb{C})},\mathtt{A})$ is a cotensor of the
$\C$-enriched category $\C/\mathbb{G}$ ($\C/\mathbb{G}$ is a
$\C$-enriched category because it is a $2$-category), because we have
the following isomorphism in $\C$

\[\C/\mathbb{G}((\mathcal{E}',\mathtt{A});(\mathcal{E}^{(\mathbb{C})},\mathtt{A}))
\simeq
Funct(\mathbb{C};\C/\mathbb{G}((\mathcal{E}',\mathtt{A});(\mathcal{E},\mathtt{A}))).\]

We also have the diagonal overfunctor (also noted $\Delta$):
$(\mathcal{E},\mathtt{A})\xrightarrow{\Delta}(\mathcal{E}^{(\mathbb{C})},\mathtt{A})$
defined by $x\longmapsto (\Delta(x),\mathtt{A}(x))$.  If
$(F,B)\in(\mathcal{E}^{(\mathbb{C})},\mathtt{A})$, an overcone of
$(F,B)$ is a morphism $\Delta(x)\xrightarrow{(\tau,b)}(F,B)$
($x\in\mathcal{E}$) of $(\mathcal{E}^{(\mathbb{C})},\mathtt{A})$,
where $\mathtt{A}(x)\xrightarrow{b}B$ is a morphism of $\mathbb{G}$.
In the same way we define overcocones.
 
It is easy to see that if $\mathbb{C}$ is connected, then every
overcone is a cone in the classical sense (respectively, every overcocone
is a cocone in the classical sense).
 
The overcategory $(\mathcal{E},\mathtt{A})$ has 
$\mathbb{C}$-overlimits (that \cite{bournpenon:moncart} calls
$(\mathbb{C})$-limits) if every
$(F,B)\in(\mathcal{E}^{(\mathbb{C})},\mathtt{A})$ has a universal
overcone $\Delta(x)\xrightarrow{(\tau,1_{B})}(F,B)$ such as
$\mathtt{A}(x)=B$, i.e if we give ourselves another overcone
$\Delta(y)\xrightarrow{(\sigma,b)}(F,B)$ (where
$\mathtt{A}(y)\xrightarrow{b}B$ is a morphism of $\mathbb{G}$) then
there is a unique morphism $y\xrightarrow{f}x$ in $\mathcal{E}$ such
that $(\tau,1_{B})\Delta(f)=(\sigma,b)$.  The definition of
$\mathbb{C}$-overcolimit is dual.  The definition of overlimits and
overcolimits enable us to include the case where $\mathbb{C}$ is the
empty category, which gives an alternative definition of 
overinitial objects (see \cref{FreydTheoremintheSurcategoricalcontext})
and overfinal objects.
 
If $\mathbb{C}$ is connected and nonempty then it is easy to see that
the following definitions are equivalent
 
\begin{itemize}
\item $(\mathcal{E},\mathtt{A})$ has $\mathbb{C}$-overlimits.
\item $\forall(F,B)\in(\mathcal{E}^{(\mathbb{C})},\mathtt{A})$,
  $(F,B)$ has a universal overcone
  $\Delta(x)\xrightarrow{(\tau,1_{B})}(F,B)$ such as
  $\mathtt{A}(x)=B$.
\item $\forall(F,B)\in(\mathcal{E}^{(\mathbb{C})},\mathtt{A})$, the
  functor $\mathbb{C}\xrightarrow{F}\mathcal{E}_{B}$ has a limit which
  is preserved by the canonical inclusion
  $\mathcal{E}_{B}\hookrightarrow\mathcal{E}$.
\item The diagonal overfunctor
  $(\mathcal{E},\mathtt{A})\xrightarrow{\Delta}(\mathcal{E}^{(\mathbb{C})},\mathtt{A})$
  has a right overadjoint.
\end{itemize}
In the same way, if $\mathbb{C}$ is connected and nonempty we have
dual definitions for $\mathbb{C}$-overcolimits.
\begin{remark}
  Let $\vec{\mathbb{N}}$ be the category of non-negative integers
  with the natural order. In the terminology adopted in
  ~\cite{bournpenon:moncart} $\vec{\mathbb{N}}$-limits are colimits. We
  prefer to adopt the word $\vec{\mathbb{N}}$-colimit for this
  specific kind of filtered colimit. And in the overcategorical context we
  prefer the word $\vec{\mathbb{N}}$-overcolimits instead of
  $(\vec{\mathbb{N}})$-colimits (as adopted by
  \cite{bournpenon:moncart}).
\end{remark}

We are now going to define $K$-equalizers and $K$-coequalizers which
are important notions because with them we get a overadjonction result
similar to Freyd's Adjoint theorem (\cref{surfreyd}), but more
general. 
 
A overcategory $(\mathcal{E},\mathtt{A})$ has $K$-equalizers if every
pair $\xymatrix{a\ar[r]<+2pt>^{f}\ar[r]<-2pt>_{g}&b}$ in $\mathcal{E}$, which has the
property $\mathtt{A}(f)=\mathtt{A}(g)$, has an equalizer $e$ in
$\mathcal{E}$
\[\xymatrix{c\ar[d]_{e}&\\
  a\ar[r]<+2pt>^{f}\ar[r]<-2pt>_{g}&b}\]
such as $\mathtt{A}(e)=\mathtt{A}(1_{a})$.  The definition of
$K$-coequalizers is dual.

If $T$ is a overmonad on $(\mathcal{E},\mathtt{A})$, the Eilenberg-Moore
algebra category $\mathcal{E}^{T}$ is trivially an overcategory
$(\mathcal{E}^{T},\mathtt{A})$ where its objects are called
overalgebras, not only to emphasise the overcategorical context, but also
to focus on the fact that an overalgebra is an algebra which lives in a
fiber.

The following propositions are immediate.

\begin{proposition}
  \label{splitsurfork}
  Let us call split overfork, a split fork in the overcategorical context,
  i.e it is a diagram
  $\xymatrix{a\ar[r]<+2pt>^{f}\ar[r]<-2pt>_{g}&b\ar[r]^{h}&c}$ which
  is a fork in a fiber $\mathcal{E}_{B}$. Then such split overforks are
  absolute overcoequalizers.
\end{proposition}
\begin{proposition}
  \label{suralgebrasourcoequalizers}
  Every overalgebra (for a fixed overmonad) is a overcoequalizer.
\end{proposition}
\begin{proposition}
  \label{surcompletsureqsurpro}
  $(\mathcal{E},\mathtt{A})$ is overcomplete iff
  $(\mathcal{E},\mathtt{A})$ has overequalizers and overproducts.
\end{proposition}
\begin{proposition}
  \label{surcompletsureqsurprobis}
  $(\mathcal{E},\mathtt{A})$ is overcocomplete iff
  $(\mathcal{E},\mathtt{A})$ has overcoequalizers and oversums.
\end{proposition}
\subsection{Some results of over(co)completeness of overalgebras}
\label{Someresultsofsurcompletudeaboutsuralgebras}
The following propositions are very similar to the classical ones and so do not require detailed
proof.
\begin{proposition}
  \label{surcompletude_des_algebres}
  Let $T$ be an overmonad on $(\mathcal{E},\mathtt{A})$. In this case:
  \[(\mathcal{E},\mathtt{A})\hspace{.1cm}\text{is
    overcomplete}\Longrightarrow(\mathcal{E}^{T},\mathtt{A})
  \hspace{.1cm}\text{is overcomplete}\]
\end{proposition}
\begin{proposition}
  \label{K-egalisateurs_dans_les_algebres}
  Let $T$ be an overmonad on $(\mathcal{E},\mathtt{A})$. In this case:
  \[(\mathcal{E},\mathtt{A})\hspace{.1cm}\text{has}\hspace{.1cm}
  K\text{-equalizers}\Longrightarrow(\mathcal{E}^{T},\mathtt{A})\hspace{.1cm}\text{has}
  \hspace{.1cm}K\text{-equalizers}\]
\end{proposition}
\begin{proposition}
  \label{coegalisateurs_entraine_cocompletude}
  Let $T$ be an overmonad on $(\mathcal{E},\mathtt{A})$. Suppose that
  $(\mathcal{E},\mathtt{A})$ is overcocomplete. In this case:
  \[(\mathcal{E}^{T},\mathtt{A})\hspace{.1cm}\text{has
    overcoequalizers}\Longleftrightarrow
  (\mathcal{E}^{T},\mathtt{A})\hspace{.1cm}\text{is overcocomplete} \]
\end{proposition}
\subsection{Freyd's Adjoint Theorem in the overcategorical context}
\label{FreydTheoremintheSurcategoricalcontext}
As we are going to see, Freyd's Adjoint Theorem remains true in the
context of $\C/\mathbb{G}$. We call it "Freyd's Overadjoint Theorem" to
refer to its overcategorical nature.  Freyd's Overadjoint Theorem can be used for example for
the proof of the \cref{SurBarrWells} which allows us to
prove some overcocompleteness results. But as we will demonstrate, unlike "Beck's
Theorem in the overcategorical context" (see
\cref{BeckTheoremintheSurcategoricalcontext}), Freyd's overadjoint
theorem requires in addition $K$-equalizers (see \cref{surfreyd}).
 
Let $(\mathcal{A},\mathtt{A})\xrightarrow{F}(\mathcal{B},\mathtt{A})$
be an overfunctor and $B\in(\mathcal{B},\mathtt{A})$.  An object of the
comma category $(B\downarrow F)$ is given by a couple $(A,a)$ 
consisting of an object $A$ of $\mathcal{A}$ and 
 to a morphism $B\xrightarrow{a}F(A)$ in $\mathcal{B}$, and
a morphism of $(B\downarrow F)$ is given by an arrow
$(A,a)\xrightarrow{f}(A',a')$ such that $F(f)a=a'$.
 
The comma category $(B\downarrow F)$ is an overcategory. Indeed we have
the arity functor $(B\downarrow
F)\xrightarrow{\mathtt{A}}\mathtt{A}(B)/\mathbb{G}$ defined on the
objects as: $(A,a)\longmapsto\mathtt{A}(a)$ and defined on the
morphism as: $f\longmapsto\mathtt{A}(f)$ ($\mathtt{A}$ is here the
arity functor of the overcategory $(\mathcal{B},\mathtt{A})$).
 
Furthermore we have the following canonical morphism of overcategories,
given by the first projection
\[\xymatrix{(B\downarrow F)\ar[d]_{\mathtt{A}}\ar[rr]^{Q}&&\mathcal{A}
  \ar[d]^{\mathtt{A}}\\
  \mathtt{A}(B)/\mathbb{G}\ar[rr]^{Q_{0}}&&\mathbb{G}}\]
\begin{proposition}
  \label{surcompletudecomma}
  Let
  $\xymatrix{(\mathcal{A},\mathtt{A})\ar[r]^{G}&(\mathcal{X},\mathtt{A})}$
  be a overfunctor such that $(\mathcal{A},\mathtt{A})$ is overcomplete and
  has $K$-equalizers. We suppose that $G$ preserves overlimits and
  $K$-equalizers. Then $\forall B\in\mathcal{X}$, the comma
  overcategory $((B\downarrow G),\mathtt{A})$ is overcomplete and has
  $K$-equalizers.
\end{proposition}
\begin{proof}
  It is enough to prove that the functor $((B\downarrow
  G),\mathtt{A})\xrightarrow{Q}(\mathcal{A},\mathtt{A})$ creates small
  overproducts, overequalizers, and $K$-equalizers. First we consider
  all functors $J\xrightarrow{F}(B\downarrow G)$ such that
  $F\in(B\downarrow G)^{(J)}$. Thus $QF\in A^{(J)}$ and if $J$ is a
  small discret category, then $lim QF$ exists
  because 
  $(\mathcal{A},\mathtt{A})$ is overcomplete.  It is easy to prove (as
  in~\citep[]{maclane:1998}) that $lim F$ exists and
  that 
  it is unique such that $Q(\lim F)=\lim QF$. If $J=\downdownarrows$, we
  use a similar argument to prove that $Q$ creates overequalizers.
  
  To prove that $Q$ creates $K$-equalizers we use a similar argument,
  but we must take $J=\downdownarrows$ and $F$ such that the image of
  the functor $\mathtt{A}F$ is a fixed arrow in
  $\mathtt{A}(B)/\mathbb{G}$.
\end{proof}
Let $(\mathcal{D},\mathtt{A})$ be an overcategory and let $G\in
\mathbb{G}$.  The object $0_{G}\in \mathcal{D}_{G}$ is overinitial if
for all objects $d\in \mathcal{D}$, and for all
$G\xrightarrow{b}\mathtt{A}(d)$ in $\mathbb{G}(1)$, there is a unique
morphism $0_{G}\xrightarrow{x}d$ of $\mathcal{D}$ over $b$.
\begin{proposition}
  \label{miracle}
  Let
  $(\mathcal{A},\mathtt{A})\xrightarrow{F}(\mathcal{B},\mathtt{A})$ be
  an overfunctor, $B\in(\mathcal{B},\mathtt{A})$, and $(R_{B},v)$ be an object
  of $((B\downarrow F),\mathtt{A})$ such that
  $\mathtt{A}(v)=1_{\mathtt{A}(B)}$. In this case:
  \[(R_{B},v)\hspace{.1cm}\text{is overinitial in }\hspace{.1cm}((B\downarrow
  F),\mathtt{A}) \Longleftrightarrow v \hspace{.1cm}\text{is initial
    in}\hspace{.1cm}(B\downarrow F)\]
\end{proposition}
\begin{lemma}[Lemma of the overinitial object]
  \label{surinitialobjet}
  Let $(\mathcal{D},\mathtt{A})$ a overcategory overcomplete with
  $K$-equalizers, and let $G\in\mathbb{G}$.
  
  In this case we have the following equivalence
  \begin{center}
    \begin{minipage}[c]{0.4\linewidth}
      $(\mathcal{D},\mathtt{A})$ has an overinitial object in one fiber
      $\mathcal{D}_{G}$
    \end{minipage}
    \quad$\Longleftrightarrow$\quad
    \begin{minipage}[c]{0.4\linewidth}
      There is a set $I$ and a family of objects
      $k_{i}\in\mathcal{D}_{G}$ ($i\in I$) such that $\forall d$ in
      $(\mathcal{D},\mathtt{A})$, $\forall
      G\xrightarrow{h}\mathtt{A}(d)$ in $\mathbb{G}$, there is an
      $i\in I$, and there is a morphism $k_{i}\xrightarrow{}d$ in
      $\mathcal{D}$ over $h$ (via the arity functor).
    \end{minipage}
  \end{center}
\end{lemma}
The proof of this lemma is very similar to the classical one
(see~\citep[proposition~1.8 page~25]{bournpenon:moncart}) and thus it
is not necessary to give the details of the demonstration. It is
useful to note that this demonstration requires $K$-equalizers.
 
Let $(\mathcal{A},\mathtt{A})\xrightarrow{F}(\mathcal{B},\mathtt{A})$
be an overfunctor.  An object $B\in (\mathcal{B},\mathtt{A})$ has the
solution set condition for $F$ if there is a set $I$ and a set of
objects $\{(A_{i},b_{i})/ i\in I\hspace{.1cm}\text{and}\hspace{.1cm}
\mathtt{A}(b_{i})=1_{\mathtt{A}(B)}\}\subset (B\downarrow F)$, such
that $\forall (A,b)\in (B\downarrow F)$, $\exists i\in I$, $\exists
A_{i}\xrightarrow{a_{i}}A$ in $(\mathcal{A},\mathtt{A})$, such that
$F(a_{i})b_{i}=b$.
\begin{theorem}[Freyd's Overadjoint theorem]
  \label{surfreyd}
  Let $(\mathcal{A},\mathtt{A})$ be an overcomplete overcategory with
  $K$-equalizers, and let
  $(\mathcal{A},\mathtt{A})\xrightarrow{F}(\mathcal{B},\mathtt{A})$ be an
  overfunctor. In that case the following properties are equivalent
  \begin{center}
    \begin{minipage}[c]{0.4\linewidth}
      $F$ has a left overadjunction
    \end{minipage}
    \quad$\Longleftrightarrow$\quad
    \begin{minipage}[c]{0.4\linewidth}
      $F$ preserve overlimits and $K$-equalizers and every object
      $B\in(\mathcal{B},\mathtt{A})$ has a solution set condition for
      $F$
    \end{minipage}
  \end{center}
\end{theorem}
\begin{proof}
  First we suppose that $F$ preserves overlimits and $K$-equalizers and
  every object $B\in(\mathcal{B},\mathtt{A})$ has a solution set
  condition for $F$. Let $B\in Ob(\mathcal{B})$, the overcategory
  $(\mathcal{A},\mathtt{A})$ is overcomplet and have $K$-equalizers
  which are preserved by $F$, thus thanks to the
  \cref{surcompletudecomma} we know that $((B\downarrow
  F),\mathtt{A})$ is overcomplet and have $K$-equalizers. Therefore
  $((B\downarrow F),\mathtt{A})$ verifies in addition the hypothesis
  "solution set condition" of the lemma of the overinitial object in
  the fiber $(B\downarrow F)_{1_{\mathtt{A}(B)}}$. Thus $((B\downarrow
  F),\mathtt{A})$ has a overinitial object in the fiber $(B\downarrow
  F)_{1_{\mathtt{A}(B)}}$. If we write down
  $B\xrightarrow{\eta_{B}}F(R_{B})$ this overinitial object, then
  thanks to the \cref{miracle}, it is initial in $(B\downarrow
  F)$. Then $F$ has a left adjoint: $G\dashv F$, and it is clearly an
  overadjoint. The converse is trivial.
\end{proof}
\subsection{A Theorem of Barr and Wells in the Overcategorical Context}
\label{ATheoremofBarrandWellsintheSurcategoricalcontext}

As we are going to see, we have an overcategorical version of the result
that we can find in \cite{francisborceux:handbook2}.  This theorem is a overcategorical
adaptation of some results that we can find
in \cite{barrwells:ttt}.
\cref{FusionProductofAdjunctions}).
\begin{theorem}[Barr-Wells's Overcategorical Theorem]
  \label{SurBarrWells}
  Let $(\mathcal{C},\mathtt{A})$ be an overcomplete and overcocomplete
  overcategory with $K$-equalizers.  Let $T$ be an overmonad on
  $(\mathcal{C},\mathtt{A})$, which preserves $\kappa$-filtered
  overcolimits for some regular cardinal $\kappa$. In this case the
  overcategory $(\mathcal{C}^{T},\mathtt{A})$ of overalgebras is
  overcomplete, overcocomplete, and has $K$-equalizers.
\end{theorem}
\begin{proof}
  The overcompletness of $(\mathcal{C}^{T},\mathtt{A})$ and the fact
  that $(\mathcal{C}^{T},\mathtt{A})$ has $K$-equalizers is a direct
  consequence of \cref{surcompletude_des_algebres} and
  \cref{K-egalisateurs_dans_les_algebres}.
  
  Thanks to \cref{coegalisateurs_entraine_cocompletude}, we also
  know that it is sufficient to prove that
  $(\mathcal{C}^{T},\mathtt{A})$ has overcoequalizers to demonstrate
  that it is overcocomplete. To prove the existence of overcoequalizers
  in $(\mathcal{C}^{T},\mathtt{A})$, it suffices to show that the
  diagonal overfunctor
  \[\xymatrix{(\mathcal{C}^{T},\mathtt{A})\ar[rr]<+6pt>^{\Delta}
    &&((\mathcal{C}^{T})^{(\downdownarrows)},\mathtt{A})\ar@{.>}[ll]<+4pt>^{\colim}_{\top}}\]
  has a left overadjoint $\colim\dashv\Delta$.  We are in a position to apply
  Freyd's Overadjoint Theorem (see \cref{surfreyd}),
  because $(\mathcal{C}^{T},\mathtt{A})$ is overcomplete and has
  $K$-equalizers and $\Delta$ preserves overlimits and $K$-equalizers.
  This last point is easy because limits in
  $(\mathcal{C}^{T})^{(\downdownarrows)}$ are computed pointwise.  We
  need to show that every object of
  $((\mathcal{C}^{T})^{(\downdownarrows)},\mathtt{A})$ has a solution
  set condition $S_{F}$ for $\Delta$. In particular if $(F,G_{0})\in
  (\mathcal{C}^{T})^{(\downdownarrows)}$, then this solution set
  condition $S_{F}$ must be in
  $\mathcal{C}^{T}_{\mathtt{A}((F,G_{0}))=G_{0}}$ \hspace{.1cm}. More
  precisely $F$ is the following data: It is a pair of morphism of
  $\mathcal{C}^{T}$:
  $\xymatrix{(A,\alpha)\ar[r]<+2pt>^{f}\ar[r]<-2pt>_{g}&(B,\beta)}$,
  which is in the fiber $\mathcal{C}^{T}_{G_{0}}$.  A solution set
  condition $S_{F}$ for $\Delta$ is given by
  
  \[S_{F}=\{\xymatrix{(B,\beta)\ar[r]^{b_{i}}&(D_{i},\delta_{i})}\in\mathcal{C}^{T}_{G_{0}}/i\in
  I\}\] 
  
  such that we give ourselves the natural transformation
  $F\xrightarrow{\sigma}\Delta(C,\gamma)$ (where
  $(C,\gamma)\in(\mathcal{C}^{T},\mathtt{A})$), then there are a
  $(D_{i},\delta_{i})\in S_{F}$, a morphism of overalgebras
  $(D_{i},\delta_{i})\xrightarrow{a}(C,\gamma)$, and a natural
  transformation $F\xrightarrow{\tau}\Delta(D_{i},\delta_{i})$, such
  that $\Delta(a)\tau=\sigma$. Therefore it means that when we
  consider the following diagram (where $j$ is a morphism of
  $\mathbb{G}$ and $h$ is not necessary in the same fiber as $f$
  and $g$; here we have $\mathtt{A}(f)=\mathtt{A}(g)=1_{G_{0}}$ and
  $\mathtt{A}(h)=j$)
  \[\xymatrix{(A,\alpha)\ar[d]<+1.4pt>_{f}\ar[d]<-1.4pt>^{g}\\
    (B,\beta)\ar[rd]^{h}\\
    &(C,\gamma)\\
    G_{0}\ar[r]^{j}&G_{1}}\]
    
  such that $hf=hg$, then $\exists i\in I$, $\exists
  (D_{i},\delta_{i})\xrightarrow{a_{i}}C,\gamma)$, such that the
  following diagram commutes
  \[\xymatrix{(A,\alpha)\ar[d]<+1.4pt>_{f}\ar[d]<-1.4pt>^{g}\\
    (B,\beta)\ar@{.>}[d]_{b_{i}}\ar[rd]^{h}\\
    (D_{i},\delta_{i})\ar@{.>}[r]_{a}&(C,\gamma)\\
    G_{0}\ar[r]^{j}&G_{1}}\]
    
  If we build such a solution set then Freyd's Overadjoint
  Theorem (see \cref{surfreyd}) shows that
  $(\mathcal{C}^{T},\mathtt{A})$ has overcoequalizers.
 
  This solution set condition is built as in the classical case
  (see~\citet[proposition~4.3.6 page~206]{francisborceux:handbook2}),
  i.e by transfinite induction, and there is no difficulty in transcribing
   it from the classical case to the overcategorical context.
\end{proof}

\subsection{Beck's theorem in the overcategorical Context}
\label{BeckTheoremintheSurcategoricalcontext}

It is easy to see that Beck theorem remains true in
$\C/\mathbb{G}$. We call this theorem "over-Beck's
theorem" to refer to its overcategorical nature. Like in the classical
case, we use two lemmas which facilitate the
demonstration of the over-Beck's theorem (see~\cite{maclane:1998}). But the proof of these two lemmas
and of the over-Beck's theorem are very similar to the classical one
(see~\cite{maclane:1998}), and thus it is not necessary to give the
details of the demonstrations. Contrary to the Freyd's
overadjoint theorem  and the Barr-Wells's overcategorical Theorem, we notice that we do not need
the presence of $K$-equalizers.
\begin{lemma}
  \label{resultat1_beck}
  Let
  $\xymatrix{(\mathcal{A},\mathtt{A})\ar[r]<+6pt>^{G}&(\mathcal{X},\mathtt{A})
    \ar[l]<+4pt>^{F}_{\top}}$,\vspace{2cm}$\xymatrix{(\mathcal{A}',\mathtt{A})\ar[r]<+6pt>^{G'}
    &(\mathcal{X},\mathtt{A})\ar[l]<+4pt>^{F'}_{\top}}$, two
  overadjunctions which generate the same overmonad $T$. If we suppose that
  $G$ satisfies hypothesis $3$ of \cref{surbeck} then there is 
  a unique overfunctor
  $(\mathcal{A}',\mathtt{A})\xrightarrow{M}(\mathcal{A},\mathtt{A})$
  such that  the following diagram commutes
  \[\xymatrix{\mathcal{X}\ar[d]_{1_{\mathcal{X}}}\ar[r]^{F'}&\mathcal{A}'
    \ar[d]^{M}\ar[r]^{G'}&\mathcal{X}\ar[d]^{1_{\mathcal{X}}}\\
    \mathcal{X}\ar[r]_{F}&\mathcal{A}\ar[r]_{G}&\mathcal{X}}
  \]
\end{lemma}
\begin{lemma}
  \label{resultat2_beck}
  In the situation
  $\xymatrix{(\mathcal{A},\mathtt{A})\ar[r]<+6pt>^{G}&(\mathcal{X}\mathtt{A})\ar[l]<+4pt>^{F}_{\top}}$,
  $G^{T}$ creates overcoequalizers of $(\mathcal{X}^{T},\mathtt{A})$
  for absolute overcoequalizers, that is given the diagram
  \[\xymatrix{(x,h)\ar[r]<+2pt>^{d_{0}}\ar[r]<-2pt>_{d_{1}}&(y,k)}\]
  in one fiber of $(\mathcal{X}^{T},\mathtt{A})$ such that the pair
  \[
  \xymatrix{G^{T}((x,h))\ar[r]<+2pt>^{G^{T}(d_{0})}
    \ar[r]<-2pt>_{G^{T}(d_{1})}&G^{T}((y,k))},\hspace{.1cm}(\text{i.e}\hspace{.1cm}
  \xymatrix{x\ar[r]<+2pt>^{d_{0}}\ar[r]<-2pt>_{d_{1}}&y})
  \]
  has an absolute overcoequalizer $y\xrightarrow{e}z$, so there is a
  unique $T$-algebra $(z,m)$ and a unique morphism
  $(y,k)\xrightarrow{f}(z,m)$ of $(\mathcal{X}^{T},\mathtt{A})$ such
  that $G^{T}(f)=e$ and furthermore $(y,k)\xrightarrow{f}(z,m)$ is a
  overcoequalizer of the pair
  $\xymatrix{(x,h)\ar[r]<+2pt>^{d_{0}}\ar[r]<-2pt>_{d_{1}}&(y,k)}$
\end{lemma}
\begin{theorem}[Beck's Overmonadicity Theorem]
  \label{surbeck}
  Let us consider the overadjunction 
  $\xymatrix{(\mathcal{A},\mathtt{A})\ar[r]<+6pt>^{G}&(\mathcal{X},\mathtt{A})\ar[l]<+4pt>^{F}_{\top}}$
  with overmonad $T$,
  the canonical final overadjunction
  $\xymatrix{(\mathcal{X}^{T},\mathtt{A})\ar[r]<+6pt>^{G^{T}}&(\mathcal{X},\mathtt{A})
    \ar[l]<+4pt>^{F^{T}}_{\top}}$, and the comparaison overfunctor
  $(\mathcal{A},\mathtt{A})\xrightarrow{K}(\mathcal{X}^{T},\mathtt{A})$
  which is the unique overfunctor such that the following diagram
  commutes
  \[\xymatrix{\mathcal{X}\ar[d]_{1_{\mathcal{X}}}\ar[r]^{F}&\mathcal{A}
    \ar[d]^{K}\ar[r]^{G}&\mathcal{X}\ar[d]^{1_{\mathcal{X}}}\\
    \mathcal{X}\ar[r]_{F^{T}}&\mathcal{X}^{T}\ar[r]_{G^{T}}&\mathcal{X}}\]
  In this case the following conditions are equivalent
  \begin{enumerate}
  \item $K$ is an isomorphism in $\C/\mathbb{G}$ (i.e there is a
    overfunctor
    $(\mathcal{X}^{T},\mathtt{A})\xrightarrow{L}(\mathcal{A},\mathtt{A})$
    such as $KL=1_{\mathcal{X}^{T}}$ and $LK=1_{\mathcal{A}}$).
  \item
    $(\mathcal{A},\mathtt{A})\xrightarrow{G}(\mathcal{X},\mathtt{A})$
    creates overcoequalizers of
    $\xymatrix{a\ar[r]<+2pt>^{f}\ar[r]<-2pt>_{g}&b}$ for which the
    pair $\xymatrix{G(a)\ar[r]<+2pt>^{G(f)}\ar[r]<-2pt>_{G(g)}&G(b)}$
    has an absolute overcoequalizer.
  \item
    $(\mathcal{A},\mathtt{A})\xrightarrow{G}(\mathcal{X},\mathtt{A})$
    creates overcoequalizers of
    $\xymatrix{a\ar[r]<+2pt>^{f}\ar[r]<-2pt>_{g}&b}$ for which the
    pair $\xymatrix{G(a)\ar[r]<+2pt>^{G(f)}\ar[r]<-2pt>_{G(g)}&G(b)}$
    has split overcoequalizers.
  \end{enumerate}
\end{theorem}

\section{Free Overmonoids} 
\label{LiberalsSurmonoidalsCategories}
In \cite{bournpenon:moncart} the authors suggest two constructions of
the free monoid associated with an object of a monoidal category.
This first construction (\citet[proposition~1.2
page~14]{bournpenon:moncart}) requires further properties on the
underlying monoidal category that the authors call "numérale" (for the
overcategorical context; see \citet[proposition~1.3.3
page~24]{bournpenon:moncart}).  The second construction of the free
monoid such as it is found in~\citet[proposition~1.3
page~16]{bournpenon:moncart} fits well with the pointed case (see
\cref{PointedSurmonoidal}) and we are especially interested in this
case (but in the overcategorical context). As the first construction,
this second construction requires further properties on the underlying
monoidal category. Therefore we call "liberal" those useful properties
by which the free monoid can be obtained from this second
construction. We shall make a small reminder of the main results but
the reader is deeply encouraged to see the details of these
constructions in~\cite{bournpenon:moncart} because we greatly use them
at the end of the proof of the \cref{theoremdufreesurmonoid}.  After
we will show that all of these constructions apply in the overmonoidal
context (which is the overcategorical version of the monoidal context),
where overmonoidals overcategories are for monoidal categories what
overcategories are for categories. Although techniques used here are close to those
 we find in \cite{bournpenon:moncart}, some concepts like
 Liberal Overmonoidal Overcategories and Pointed Overmonoidal Overcategories are new.
 In particular the proof of the \cref{theoremdufreesurmonoid} 
 is similar to the \cref{theoremdufreemon} below which is in \cite{bournpenon:moncart}.
\subsection{Liberal monoidal categories}
\label{LiberalMonoidalCategories}
Let $\mathcal{V}=(\mathbb{V},\bigotimes,I,u_{l},u_{r},ass)$ be a
monoidal category. We sometimes denote it by its underlying category
$\mathbb{V}$.  $\mathcal{V}$ is liberal if the following properties
hold:
\begin{itemize}
\item $\mathbb{V}$ has $\overrightarrow{\mathbb{N}}-$colimits and
  coequalizer;
\item $\forall{X\in{\mathbb{V}(0)}}$, $(-)\bigotimes{X}$ and
  $X\bigotimes{(-)}$ preserves $\overrightarrow{\mathbb{N}}-$colimits;
\item $\forall{X\in{\mathbb{V}(0)}}$, $(-)\bigotimes{X}$ preserves
  coequalizers.
\end{itemize}
Let $\mathbb{M}on(\mathbb{V})$ be the category of monoids in $\mathbb{V}$. We have a forgetful functor
$\mathbb{M}on(\mathbb{V})\xrightarrow{U}\mathbb{V}$,
$(M,e,m)\longmapsto M$ and we have in \citet[proposition~1.3
page~16]{bournpenon:moncart}:

\begin{proposition}        
  \label{theoremdufreemon}
  If $\mathcal{V}$ is liberal and if $I$ is an initial object then the
  preceding forgetful functor has a left adjoint
  \[\xymatrix{\mathbb{M}on(\mathbb{V})\ar[r]<+6pt>^(.6){U}&\mathbb{V}\ar[l]<+4pt>^(.4){Mo}_(.35){\top}}\]
\end{proposition}
In order to construct this free monoid functor $\Mo(-)$, we use the
notion of graded monoid (defined in
\citet[page~12]{bournpenon:moncart}). A graded monoid in a monoidal
category $\mathcal{V}$ is given by a triple
$((X_n)_{n\in{\mathbb{N}}},(\iota_n)_{n\in{\mathbb{N}}},(k_{n,m})_{n,m\in{\mathbb{N}}})$
where $(X_n)_{n\in{\mathbb{N}}}$ is a family of objects of
$\mathcal{V}$,
$(X_{n}\xrightarrow{\iota_n}X_{n+1})_{n\in{\mathbb{N}}}$ is a family
of morphisms of $\mathcal{V}$, and
$(X_{n}\bigotimes{X_{n}}\xrightarrow{k_{n,m}}X_{n+m})_{n,m\in{\mathbb{N}}}$
is a family of morphisms of $\mathcal{V}$, verifying some axioms that
we can find in \citet[page~12]{bournpenon:moncart}.  In
\cite{bournpenon:moncart} it is proved that every monoid has an
underlying graded monoid and every graded monoid
$((X_n)_{n\in{\mathbb{N}}},(\iota_n)_{n\in{\mathbb{N}}},(k_{n,m})_{n,m\in{\mathbb{N}}})$
is linked with a free monoid. Then the strategy to built the free
monoid $\Mo(X)$ for every $X\in{\mathbb{V}}$ is first to built a graded
monoid $\Psi_{X}$ where this construction also requires the construction by induction
of a secondary family of morphisms
$(X\bigotimes{X_{n}}\xrightarrow{q_{n}}X_{n+1})_{n\in{\mathbb{N}}}$,
then $\Mo(X)$ is also the free monoid associated with this graded
monoid
$\Psi_{X}=((X_n)_{n\in{\mathbb{N}}},(\iota_n)_{n\in{\mathbb{N}}},(k_{n,m})_{n,m\in{\mathbb{N}}})$.
 
We must remember that $\Mo(X)=\colim{X_{n}}$ and
$(X_n)_{n\in{\mathbb{N}}}$ is built by induction with morphisms
$X_{n}\xrightarrow{\iota_n}X_{n+1}$,
$X\bigotimes{X_{n}}\xrightarrow{q_{n}}X_{n+1}$, by considering the
coequalizer $q_{n+1}:=coker(y^0_n,y^1_n)$, where
\begin{itemize}
\item$y^0_n=Id\otimes{\iota_{n}}$,
\item$y^1_n=(\xymatrix{X\bigotimes{X_{n}}\ar[r]^{q_{n}}&X_{n+1}\ar[r]^{u^{-1}_{l}}
    &I\bigotimes{X_{n+1}}\ar[r]^{!\otimes{Id}}&X\bigotimes{X_{n+1}}})$,
\item$\iota_{n+1}=(\xymatrix{X_{n+1}\ar[r]^{u^{-1}_{l}}&I\bigotimes{X_{n+1}}
    \ar[r]^{!\otimes{Id}}&X\bigotimes{X_{n+1}}\ar[r]^{q_{n+1}}&X_{n+2}})$,
\end{itemize}
and where the initialization is given by $X_{0}=I, X_{1}=X$,
$I\xrightarrow{\iota_{0}=!_{X}}X$,\\
$X\bigotimes{I}\xrightarrow{q_{0}=u_{r}}X$.
 
Morphims $k_{n,m}$ are built by induction (see \citet[page~16 and
page~17]{bournpenon:moncart}), but we do not describe it here because
we do not explicitly need them anymore.  Let
$(X_{n}\xrightarrow{l_{n}}\Mo(X))_{n\in{\mathbb{N}}}$, the universal
cocone defining $\Mo(X)$. The associated universal arrow is
$X\xrightarrow{l(X)=l_{1}}\Mo(X)$.  Let us remind that the
multiplication $\Mo(X)\bigotimes{\Mo(X)}\xrightarrow{m}\Mo(X)$ is the
unique arrow such as
$\forall{n,m\in{\mathbb{N}}}$:\hspace{.1cm}$m(l_{n}\otimes
l_{m})=l_{n+m}k_{n,m}$.  When $n=1$ we have $k_{1,m}=q_{m}$ which
gives the equality $m(l_{1}\otimes l_{m})=l_{m+1}q_{m}$ and which will
be useful for the construction of the free overmonoid (see
\cref{result2}).
\subsection{Liberal monoidal overcategories}
\label{contextesurmo}
Let $\mathbb{G}$ be some fixed category.

We shall expand further on the "overmonoidal" context, what we have
made for the monoidal context. 
\begin{remark}
As application we will see in \cite{kach:weak_des_weak}
that the results of $\C$, which has enabled to build the free
contractible operad of weak $\omega$-categories of Batanin (see
\cite{bat:monglob}) are true in $\C/\IG$, which gives us many kind of
free colored operads and especially the free contractible colored operads for weak higher transformations.
\end{remark}
   
Let us now briefly recall the definition of monoidal overcategory.
   
Let $\mathbb{G}$ be a fixed category. An monoidal overcategory (over $\mathbb{G}$) is a
monoidal object of the $2$-category $\C\diagup{\mathbb{G}}$.  A
monoidal overcategory is thus given by a $7$-uple:\hspace{.1cm}
$\mathcal{E}=(\mathbb{E},\mathtt{A},\bigotimes,I,u_{l},u_{r},ass)$
where:
\begin{itemize}
\item $\mathtt{A}$ is a functor:
  $\mathbb{E}\xrightarrow{\mathtt{A}}\mathbb{G}$;
\item
  $(\mathbb{E}\times_{\mathbb{G}}\mathbb{E},\mathtt{A})\xrightarrow{\bigotimes}(\mathbb{E},\mathtt{A})$
  is a morphism of $\C\diagup{\mathbb{G}}$, where
  $\mathbb{E}\times_{\mathbb{G}}\mathbb{E}$ is the kernel pair of $\mathtt{A}$;
\item $\mathbb{G}\xrightarrow{I}\mathbb{E}$ is a functor and a section
  (i.e we have $\mathtt{A}I=1_{\mathbb{G}}$);
\item $u_{r}$ and $u_{l}$ are natural isomorphisms:
  $\bigotimes(1_{\mathbb{E}},I\mathtt{A})\xrightarrow{u_{r}}1_{\mathbb{E}}$,
  $\bigotimes(I\mathtt{A},1_{\mathbb{E}})\xrightarrow{u_{l}}1_{\mathbb{E}}$;
\item $ass$ is a natural isomorphism: $\bigotimes(\bigotimes\times
  1_{\mathbb{E}})\xrightarrow{aso}\bigotimes(1_{\mathbb{E}}\times\bigotimes)$.
\end{itemize}
And these data satisfy the usual conditions of coherence i.e those
given by the axioms of monoidal categories.  A simple consequence of
this definition is that for every object $B$ of $\mathbb{G}$ each
fiber $\mathbb{E}_{B}$ is a monoidal category.  We write with the
same notation in each fiber the tensor product because the context
will prevent any confusion.
\begin{remark}
  Obviously, strict monoidal overcategories means that
  $u_{l},u_{r}$ and $ass$ are natural identities.
\end{remark}
Let $\mathcal{E}=(\mathbb{E},\mathtt{A},\bigotimes,I,u_{l},u_{r},ass)$
and
$\mathcal{E}'=(\mathbb{E}',\mathtt{A}',\bigotimes',I',u_{l}',u_{r}',ass')$
be two monoidal overcategories with respective base categories
$\mathbb{G}$ and $\mathbb{G}'$.  A strict morphism
$\mathcal{E}\xrightarrow{(F,F_{0})}\mathcal{E}'$, is given by two
functors $\mathbb{E}\xrightarrow{F}\mathbb{E}'$ and
$\mathbb{G}\xrightarrow{F_{0}}\mathbb{G}'$ such that
$F_{0}\mathtt{A}=\mathtt{A}'F$, $FI=I'F_{0}$ and
$F\bigotimes=\bigotimes'(F\times_{F_{0}}F)$.
Let $\mathcal{E}$ be a monoidal overcategory. An
overmonoid in $\mathcal{E}$ is given by a pair
$(\mathcal{C};C_{0})$ where $C_{0}\in{\mathbb{G}(0)}$ and
$\mathcal{C}=(C,m,e)$ is a monoid in $\mathbb{E}_{C_{0}}$ ($m$ is the
multiplication and $e$ is the unity). Thus $(\mathcal{C};C_{0})$ is more properly written as
 $(C,m,e;C_{0})$. It is a monoid in a fibre.
  
If $(\mathcal{C};C_{0})$ and $(\mathcal{C}';C_{0}')$ are overmonoids, a
morphism
\[(\mathcal{C};C_{0})\xrightarrow{(f,f_{0})}(\mathcal{C}';C_{0}'),\]
is given by a pair $(f,f_{0})$ where $C_{0}\xrightarrow{f_{0}}C_{0}'$
is an arrow in $\mathbb{G}$ and
$\mathcal{C}\xrightarrow{f}\mathcal{C}'$ is given by an arrow
$C\xrightarrow{f}C'$ in $\mathbb{E}$ such as $\mathtt{A}(f)=f_{0}$ and
$fm=m'(f\otimes_{f_{0}}f)$, $fe=e'I(f_{0})$. We note
$/\mathbb{M}on(\mathbb{E},\mathtt{A})$ the category of overmonoids of 
$\mathcal{E}$.
  
Let $\mathcal{E}$ be a monoidal overcategory. It is
liberal if the following two conditions are satisfied
\begin{itemize}
\item $\forall B\in{\mathbb{G}(0)}$, the fiber $\mathbb{E}_{B}$ is a
  liberal monoidal category.
\item $\forall B\in{\mathbb{G}(0)}$, the canonical inclusion functor
  $\mathbb{E}_{B}\hookrightarrow\mathbb{E}$ preserves coequalizer and
  $\overrightarrow{\mathbb{N}}$-colimits.
\end{itemize}
Let $(\mathcal{C};C_{0})$ be a overmonoid of $\mathcal{E}$,
then
\begin{proposition}
  \label{contextesurmolasliceestsurmo}
  The pair $(\mathbb{E}\diagup{C},\widehat{\mathtt{A}})$, such as
  $\mathbb{E}\diagup{C}\xrightarrow{\widehat{\mathtt{A}}}\mathbb{G}\diagup{\mathtt{A}(C)}$,
  $x\longmapsto\mathtt{A}(x)$, produces a overmonoidal
  overcategory 
  \[\mathcal{E}\diagup{C}=(\mathbb{E}\diagup{C},\widehat{\mathtt{A}},\widehat{\bigotimes},\widehat{I},
  \widehat{u}_{l},\widehat{u}_{r},\widehat{aso})\]
\end{proposition}
 
The proof is in~\citet[page~22]{bournpenon:moncart} but let us recall
that if $(X,x),(Y,y)\in\mathbb{E}\diagup{C}$ then
$(X,x)\widehat{\bigotimes}(Y,y):=(X\bigotimes Y,m(x\otimes y))$.  If
$b\in\mathbb{G}\diagup{\mathtt{A}(C)}$ then $\widehat{I}(b):=eI(b)$.
The $2$-cells $\widehat{u}_{l}$, $\widehat{u}_{r}$, $\widehat{aso}$ are
also provided with the corresponding data of $\mathcal{E}$.
 
When overcoequalizers exist in $\mathcal{E}$, it is not difficult to
see that overcoequalizers in $\mathbb{E}\diagup{C}$ are computed by it,
and we have the same phenomenon for
$\overrightarrow{\mathbb{N}}$-overcolimits.  So we have the following
easy proposition that is left for the reader.
\begin{proposition}
  \label{contextesurmolasliceestliberal}
  If $\mathcal{E}=(\mathbb{E},\mathtt{A},\bigotimes,I,u_{l},u_{r},aso)$
  is a liberal monoidal overcategory then
  $\mathcal{E}\diagup{C}=(\mathbb{E}\diagup{C},\widehat{\mathtt{A}},\widehat{\bigotimes},\widehat{I},
  \widehat{u}_{l},\widehat{u}_{r},\widehat{aso})$ is a liberal
  monoidal overcategory, and the morphism
  $\mathcal{E}\diagup{C}\xrightarrow{(S,S_{0})}\mathcal{E}$ given by
  the functor $\mathbb{E}\diagup{C}\xrightarrow{S}\mathbb{E}$,
  $(X,x)\longmapsto X$, is a strict morphism of overmonoidal
  overcategories which preserves the liberal structure.
\end{proposition}
We have the following proposition too
\begin{proposition} 
  \label{strictmorphismofovermonoidalcategories}
  If $\mathcal{E}$ is a liberal monoidal overcategory and if \\
  $(\mathcal{C};C_{0})\xrightarrow{(h,h_{0})}(\mathcal{C}';C_{0}')$ is
  a morphism of overmonoids, then the morphism
  \[\xymatrix{\mathcal{E}\diagup{C}\ar[r]^{(h^*,h_{0}^*)}&\mathcal{E}\diagup{C'}}\]
  is a strict morphism of monoidal overcategories which preserves
  the liberal structure.
\end{proposition}
\begin{proof}
  The fact that $(h^*,h_{0}^*)$ is a strict morphism of overmonoidal
  overcategories has already been
  shown (\citet[page~25]{bournpenon:moncart}) and the fact that $h^*$
  preserves $\overrightarrow{\mathbb{N}}$-overcolimits has already been
  proved for the numeral context~\citep[see][page
  25]{bournpenon:moncart}.  We only have to show that $h^*$ preserves
  overcoequalizers and it is evident by construction.
\end{proof}
 
Now we have enough material to give the main theorem of this paragraph.
\begin{theorem}
  \label{theoremdufreesurmonoid}
  Let
  $\mathcal{E}=(\mathbb{E},\mathtt{A},\bigotimes,I,u_{l},u_{r},aso)$
  be a liberal monoidal overcategory such that $\forall
  B\in\mathbb{G}(0)$ the object $I(B)$ is initial in the fiber
  $\mathbb{E}_{B}$ and such that $\forall b\in\mathbb{G}(1)$ the object
  $I(b)$ is initial in the fiber $\mathbb{E}_{b}$, then the forgetful
  overfunctor
  \[\xymatrix{(/\mathbb{M}on(\mathbb{E},\mathtt{A}),\mathtt{A})\ar[r]^(.6){U}&(\mathbb{E},\mathtt{A})}\] 
  has a left overadjoint $M\dashv U$ and it is overmonadic.
\end{theorem}
\begin{proof}
 It is similar to the proof of \cref{theoremdufreemon} and we just need to adapt it to the
 overcategorical context. In particular we use \cref{theoremdufreemon}, 
 the reminders in \cref{LiberalMonoidalCategories}, \cref{strictmorphismofovermonoidalcategories}
 plus the following two results (The first result below is a refinement of
  \cref{strictmorphismofovermonoidalcategories}. We prove these two results by induction):
  \begin{result}
    \label{result1}
    Let
    $(\mathcal{C};C_{0})\xrightarrow{(h,h_{0})}(\mathcal{C}';C_{0}')$
    be a morphism of overmonoids, then if
    $(X,x)\in\mathbb{E}\diagup{C}$ then $\forall n\in\mathbb{N}$,
    $h^*((X,x)_{n})=(h^*((X,x)))_{n}$, where $(X,x)_{n}$ is the
    $n^{\text{th}}$ object of the graded monoid associated with
    $(X,x)$ (see \citep[1.2.3 page 12]{bournpenon:moncart} for
    definition and results about graded monoids).
  \end{result}
  \begin{result}
    \label{result2}
    $\forall n\in\mathbb{N}$, $(X,l(X))_{n}=(X_{n},l_{n})$, where
    $X_{n}\xrightarrow{l_{n}}\Mo(X)$ is an arrow of the colimit cocone
    defining $\Mo(X)$, and where $(X,l(X))_{n}$ is the $n^{\text{th}}$
    object of the graded monoid associated with $(X,l(X))$.
  \end{result}
 
  The overmonadicity of $U$ is a simple consequence of \cref{surbeck}.
  In particular this overmonadicity has already been proved in the
  numeral context~\citep[see proposition 1.3.1, page
  20]{bournpenon:moncart}.
 \end{proof}
 
Now we can study the important case of pointed overmonoidal
overcategories.
\label{PointedSurmonoidal}
In \cite{bournpenon:moncart} it is proved that to any monoidal
category $\mathcal{V}=(\mathbb{V},\bigotimes,I,u_{l},u_{r},aso)$ we
associate its pointed monoidal category
\[Pt(\mathcal{V})=(Pt(\mathbb{V}),\widetilde{\bigotimes},\widetilde{I}
,\widetilde{u}_{l},\widetilde{u}_{r},\widetilde{aso})\] and if
$\mathcal{V}$ was liberal then $Pt(\mathcal{V})$ remained liberal.

We can expand to the overmonoidal context this construction and this
result.  Let
$\mathcal{E}=(\mathbb{E},\mathtt{A},\bigotimes,I,u_{l},u_{r},aso)$ be
a monoidal overcategory over a fixed category $\mathbb{G}$.

Let $Pt(\mathbb{E})$ the category with objects the pairs $(X,x)$ where
$X\in{\mathbb{E}(0)}$ and $I(\mathtt{A}(X))\xrightarrow{x}X\in\mathbb{E}(1)$, and which has for arrows $(X,x)\xrightarrow{f}(Y,y)$,
given by morphism $X\xrightarrow{f}Y$ of $\mathbb{E}$ such as
$fx=yI_{\mathtt{A}(f)}$. In this case we have
\begin{proposition}
  \label{laPointedestSurmonoidal}
  The pair $(Pt(\mathbb{E}),\widetilde{\mathtt{A}})$ such that:
  $\xymatrix{(X,x)\ar@{|->}[r]^{\widetilde{\mathtt{A}}}&\mathtt{A}(X)}$,
  produces a structure of monoidal overcategory
  \[Pt(\mathcal{E})=(Pt(\mathbb{E}),\widetilde{\mathtt{A}},\widetilde{\bigotimes},\widetilde{I},\widetilde{u}_{l},\widetilde{u}_{r},\widetilde{aso})\]
\end{proposition}
\begin{proof}
  \begin{itemize}
  \item Its tensor is the bifunctor
    $Pt(\mathbb{E})\times_{\mathbb{G}}Pt(\mathbb{E})\xrightarrow{\widetilde{\bigotimes}}Pt(\mathbb{E})$,\\
    $((X,x),(Y,y))\longmapsto
    (X,x)\widetilde{\bigotimes}(Y,y):=(X\bigotimes Y,(x\otimes
    y)u^{-1}_{l})$.
  \item Its "unity" functor is
    $\mathbb{G}\xrightarrow{\widetilde{I}}Pt(\mathbb{E})$,
    $G\longmapsto (I(G),1_{I(G)})$.
  \item Left and right isomorphisms of unity: For all $(X,x)$ of
    $Pt(\mathbb{E})(0)$ the tensor
    $\widetilde{I}(\widetilde{\mathtt{A}}(X,x))
    \widetilde{\bigotimes}(X,x)$ is given by the morphism
    $(1_{I(\mathtt{A}(X))}\otimes x)u^{-1}_{l}$ of $\mathbb{E}$, and
    we have $u_{l}(1_{I(\mathtt{A}(X))}\otimes x)u^{-1}_{l}=x$ thanks
    to the equality
    \[u_{l}(X)(1_{I(\mathtt{A}(X))}\otimes x)
    =xu_{l}(I(\mathtt{A}(X))).\] Thus we get
    \[\xymatrix{\widetilde{I}(\widetilde{\mathtt{A}}(X,x))\widetilde{\bigotimes}(X,x)
      \ar[rr]^(.6){\widetilde{u}_{l}(X,x)}&&(X,x)}\]
    and $\widetilde{u}_{l}(X,x)$ given by $u_{l}(X)$ is a good
    candidate to define $\widetilde{u}_{l}$.  Thus we obtain the
    natural transformation
    $\widetilde{\bigotimes}(\widetilde{I}\widetilde{\mathtt{A}},Id)\xrightarrow{\widetilde{u}_{l}}Id$
    which is in fact, an underlying datum of its $2$-cell
    $\widetilde{u}_{l}$.  In the same way we obtain the $2$-cell
    $\xymatrix{\widetilde{\bigotimes}(Id,\widetilde{I}\widetilde{\mathtt{A}})
      \ar@{=>}[r]^(.6){\widetilde{u}_{r}}&Id}$.
  \item The tensor products
    \[((X,x)\widetilde{\bigotimes}(Y,y))\widetilde{\bigotimes}(Z,z) 
    \hspace{.1cm}\text{and}\hspace{.1cm}(X,x)\widetilde{\bigotimes}((Y,y)\widetilde{\bigotimes}
    (Z,z))\] are respectively given by
    \[[((x\otimes y)u^{-1}_{l})\otimes
    z]u^{-1}_{l}\hspace{.1cm}\text{and}\hspace{.1cm}[x\otimes((y\otimes
    z)u^{-1}_{l})]u^{-1}_{l},\] and we have the equality
    \[aso[((x\otimes y)u^{-1}_{l})\otimes
    z]u^{-1}_{l}=[x\otimes((y\otimes z)u^{-1}_{l})]u^{-1}_{l}\] due to
    the naturality of $aso$ and the underlying overmonoid structure of
    $I(\mathtt{A}(X))$.  We consequently obtain
    \[\xymatrix{((X,x)\widetilde{\bigotimes}(Y,y))\widetilde{\bigotimes}(Z,z)\ar[r]^{\widetilde{aso}}&
      (X,x)\widetilde{\bigotimes}((Y,y)\widetilde{\bigotimes}(Z,z))}\]
    where in particular $\widetilde{aso}$ is given by $aso$, and is
    the good candidate to be the $2$-cells of associativity.  Thus we
    obtain the natural transformation
    $\widetilde{\bigotimes}(\widetilde{\bigotimes}\times
    Id)\xrightarrow{\widetilde{aso}}
    \widetilde{\bigotimes}(Id\times\widetilde{\bigotimes})$ which in
    reality is an underlying datum of the $2$-cell $\widetilde{aso}$,
    and with this description of $Pt(\mathcal{E})$ it is now easy to
    see that it is a monoidal overcategory.
  \end{itemize}
\end{proof}
As for $\mathcal{E}\diagup{C}$, when overcoequalizers exist in
$\mathcal{E}$, then we can see that overcoequalizers in
$Pt(\mathcal{E})$ are computed by it, and we have the same phenomenon
for $\overrightarrow{\mathbb{N}}$- overcolimits. So we have the
following easy proposition.
\begin{proposition}
  \label{laPointedestLiberal}
  If
  $\mathcal{E}=(\mathbb{E},\mathtt{A},\bigotimes,I,u_{l},u_{r},aso)$
  is a liberal monoidal overcategory then
  $Pt(\mathcal{E})=(Pt(\mathbb{E}),\widetilde{\mathtt{A}},\widetilde{\bigotimes},\widetilde{I},\widetilde{u}_{l},\widetilde{u}_{r},\widetilde{aso})$
  stays a liberal overmonoidal category, and trivially the functor
  $\widetilde{I}$ send objects and arrows of $\mathbb{G}$ to initial
  objects in the corresponding fibers.
\end{proposition}
The following proposition is easy. 
It is the overmonoidal version of the result in~\citet[1.2.1 page 10]{bournpenon:moncart}. 
\begin{proposition}
  \label{propsurlesisodesurmo}
  If
  $\mathcal{E}=(\mathbb{E},\mathtt{A},\bigotimes,I,u_{l},u_{r},aso)$
  is a monoidal overcategory, then we have the commutative triangle
  \[\xymatrix{/\mathbb{M}on(\mathbb{E},\mathtt{A})\ar[rd]_U\ar[r]^(.45){\varphi}  
    &/\mathbb{M}on(Pt(\mathbb{E}),\widetilde{\mathtt{A}})\ar[d]^{U'}\\
    &Pt(\mathbb{E})}\]
  such that $\varphi$ is an isomorphism given by
  $\varphi((C,e,m;C_{0}))=((C,e),e,m;C_{0})$, and with
  $U((C,e,m;C_{0}))=(C,e)$ and $U'(((C,x),e,m;C_{0}))=(C,e)$.
\end{proposition}
With the theorem and the previous propositions we have at once:
\begin{theorem}
  \label{theoremdufreesurmoversuspointe}
  If
  $\mathcal{E}=(\mathbb{E},\mathtt{A},\bigotimes,I,u_{l},u_{r},aso)$
  is a liberal monoidal overcategory then the forgetful overfunctor
  \[\xymatrix{(/\mathbb{M}on(\mathbb{E},\mathtt{A}),\mathtt{A})          
      \ar[r]^(.55){U}&(Pt(\mathbb{E}),\widetilde{\mathtt{A}})},
  \xymatrix{(C,e,m;C_{0})\ar@{|->}[r]&(C,e)}\] has a left overadjoint
  and is overmonadic.
\end{theorem}
\begin{remark}
  \label{SystemesDOperations}
  Let us denote by $M$ the left overadjoint of $U$, then if we applied
  $(X,x)\in Pt(\mathbb{E})$ to the unity
  $1_{Pt(\mathbb{E})}\xrightarrow{l}UM$ of this overadjunction, we
  obtain the morphism $(X,x)\xrightarrow{l((X,x))}U(M(X,x))$ of
  $Pt(\mathbb{E})$ i.e
  $(X,x)\xrightarrow{l((X,x))}U(\overline{X},e,m;X_{0})=(\overline{X},e)$.
  And in particular this morphism gives us the equality $l((X,x))x=e$.
  This equality is important because it shows, in the particular
  context of colored operads of \cite{kach:nscellsfinal} and \cite{kach:weak_des_weak}, that the
  operads of weak higher transformations are well-provided with a system of
  operations.
\end{remark}

\vspace{1cm}

  \bigbreak{}
  \begin{minipage}{1.0\linewidth}
    Camell \textsc{Kachour}\\
    Macquarie University, Department of Mathematics\\
    Phone: 00612 9850 8942\\
    Email:\href{mailto:camell.kachour@mq.edu.au}{\url{camell.kachour@mq.edu.au}}
  \end{minipage}
  
\end{document}